\documentclass[pdflatex,sn-mathphys]{sn-jnl}
\usepackage{amsmath,amsthm,amssymb}
\usepackage{algorithm}
\usepackage{algpseudocode}
\usepackage{amsmath}
\usepackage{float}
\usepackage{mathtools}
\usepackage{microtype}
\usepackage{booktabs} 
\usepackage{hyperref}
\usepackage{bbm}
\newcommand{\ishy}[1]{{\color{black}#1}}
\newcommand{\ishyComments}[1]{{\color{black}#1}}
\newcommand{\ngg}[1]{{\color{black}#1}}
\usepackage{comment}

\usepackage[table]{xcolor}
\usepackage{graphicx}
\graphicspath{{./images}}

\usepackage{enumerate}
\usepackage[inline]{enumitem}
\usepackage{colortbl}
\newcommand{\real}{\mathbb{R}}

\newcommand{\vectornorm}[1]{\| #1\|}

\DeclarePairedDelimiterX\Set[1]\{\}{
	
	#1
}

\newcommand{\set}[2]{\left\{#1 \; \left|\;\; #2 \right.\right\}}

\DeclareMathOperator{\prox}{prox}

\newcommand{\rev}[1]{{\color{black}#1}}

\newcommand{\ngrev}[1]{{\color{black}#1}}
\newcommand{\ishyone}[1]{{\color{black}#1}}
\newcommand{\noamr}[1]{{\color{black}#1}}
\renewcommand{\twocolumn}[1]{#1}
\newtheorem{theorem}{Theorem}

\newtheorem{Case}{Case}

\newtheorem{lemma}[theorem]{Lemma}

\title{Projecting onto a Capped Rotated Second-Order Cone}

\author[1]{\fnm{Noam} \sur{Goldberg}}\email{goldnoam@bgu.ac.il}
\affil[1]{\orgname{Department of Industrial Engineering \& Management, Ben-Gurion University of the Negev}, \city{Beer Sheva}, \country{Israel}}

\author[2]{\fnm{Ishy} \sur{Zagdoun}}\email{ishai.zigdon@biu.ac.il}
\affil[2]{\orgname{Department of Mathematics, Bar-Ilan University}, \city{Ramat Gan,} \country{Israel}}

\begin{document}

	\abstract{ 
 We derive a closed-form expression for the projection onto a capped rotated second-order cone—a convex set that arises in perspective relaxations of nonlinear programs with binary indicator variables. The closed-form solution involves four distinct cases, one of which \ishyComments{coincides with the} classical projection onto a second-order cone. The remaining three cases \ishyComments{involve} a nontrivial 
 analysis in determining the necessary and sufficient conditions under which the solution lies on the intersection of the cone and a facet of a box. 
{Beyond theoretical interest, this result enables the development of efficient projection-based first-order methods for optimization problems involving perspective reformulations. As an illustration, we consider the continuous perspective relaxation of sparse logistic regression, whose feasible region contains a Cartesian product of capped rotated second-order cones. 
Computational experiments on synthetic datasets demonstrate that 
\ngrev{accelerated projected-gradient, together with computing our closed-form} projection 
\ngrev{at each iteration, 
improves on the running time of a general first-order conic solver while obtaining solutions of similar accuracy. }
These results highlight the practical value of the proposed projection formula for solving continuous relaxations of perspective formulations of mixed-integer 
second-order cone problems.}

}
	\maketitle
\newpage
	\section{Introduction}
    \ishyComments{Projection-based methods are fundamental tools for solving convex feasibility and constrained optimization problems. Their efficiency depends critically on the availability of inexpensive projection operators. Classical projection algorithms and their convergence properties are comprehensively surveyed by~\cite{bauschke1996proj}}.
	The current paper develops a closed-form solution for the projection onto a capped rotated second-order cone (RSOC), an intersection of a rotated second-order cone and a particular box, a set that 
	arises when solving tight continuous relaxations of convex optimization problems with binary indicator variables. 
    \ishyComments{Such relaxations arise from the perspective formulation}~\cite{AktuRk09,Gunluk10,Frangioni2011}, whose continuous relaxation is known as the \emph{perspective relaxation}. This continuous relaxation has also been proposed in the context of sparse statistical learning problems, where it is known as the Boolean relaxation~\cite{Pilanci2015}, and it has been further generalized recently in the study of~\cite{Bertsimas2023}, which uses the matrix perspective to solve matrix optimization problems {for} which low-rank solutions are desired. \ishyComments{For a positive integer $k$, we use the notation $[k]=\{1,\ldots,k\}$.} Letting $m, n_1, \dots, n_m$ be positive integers, 
 $X\subseteq\real^{n_1}\times\cdots\times\real^{n_m}\times\{0,1\}^m$, and letting $q(x,y,z)$ be a closed convex function, the perspective formulation involves \noamr{mixed-integer optimization (MIO) problems} of the form 
	\begin{equation}
		\label{prob:mio}
		\min_{(x_1,\ldots,x_m, z)\in X,\noamr{y\in\real^m_+}} \set{\noamr{q(x,y,z)}}{\noamr{\vectornorm{x_i}^2\leq y_iz_i}, \; i\in [m]}.
	\end{equation} 
	While problem \eqref{prob:mio} is stated in a general form, 
    typically, in sparse optimization problems, \ishyComments{where a sparse solution vector $x$ is desirable}, $q$ \ishyComments{may} be linear in $z$, having a penalty term such as $\lambda\sum_{i=1}^mz_i$, for some $\lambda>0$. Alternatively, the decision variable vector $z$ may not appear in the objective, in which case 
    an objective penalty is 
    replaced by a constraint in terms of the $z$ variables, such as a simple cardinality constraint, $\sum_{i=1}^mz_i\leq k$, for some integer $k$.

		\ishyComments{Rapid computation of projections can serve as an important building block for efficient first-order optimization methods.} 
		For example, 
		simple box projections have been widely adopted and have proven most effective for general large-scale bound-constrained nonlinear programming~\cite{More1991,Hager2006} and for solving SVMs~\cite{Bottou2007}. 
        \ishyComments{The importance of inexpensive computational building blocks has also been recognized more broadly in sparse optimization. For example, \cite{Luss2013} proposed efficient conditional-gradient methods for sparse principal component analysis, demonstrating that carefully designed low-cost iterations can substantially improve scalability. Our work shares the same motivation by deriving a closed-form projection operator that can be used as an efficient building block within projection-based algorithms for second-order conic optimization problems.}
        
        Second-order conic constraints 
		essentially generalize nonnegativity constraints and correspond to a special case of semidefinite constraints.	Handling semidefinite constraints may not be tractable in practice. 
        In the context of first-order methods the source of intractability is that projecting onto the $n\times n$-dimensional semidefinite cone involves diagonalizing a matrix, which is of order $O(n^3)$. In sharp contrast to this challenge,  projection onto an 
		SOC is given by an $O(n)$ closed-form solution (see \cite{Alizadeh03}, and original proof in \cite[Chapter~3]
        {Bauschke96}).  
		Projection-based methods that exploit this specific closed-form result 
        range from projected gradient~\cite{goldberg15} to 
        dual-based decomposition methods~\cite{Yang2015,ODonoghue_Boyd_2016}. 
        Projection-based and more general proximal methods in optimization include a variety of first-order methods; see~\cite{Dvurechensky2021} for a recent overview and~\cite{Beck17} for a more detailed account. Moreover, some second-order methods may also involve projection onto cones; see~\cite{henrion2012projection}.
        

		While deriving the projection onto a box or affine set is a standard textbook exercise, several studies have examined projections onto 
        more elaborate convex sets. 
        This includes projection onto a unit simplex~\cite{Condat2016,Duchi2008}, projection onto hyperbolas \cite{bauschke23}, and projections onto different generalizations of second-order cones, such as the extended SOC~\cite{Nemeth2015,Ferreira2018}. Note that even for sets for which the projection is given by a closed form, the projection onto the intersection of such sets is not generally given by a closed form. For example, this is generally the case in the study by \cite{bauschke18}, which examines projections onto the intersection of standard cones with balls or spheres, each of which possesses a simple closed-form solution. 
        \ishyComments{The main contribution of this paper is a closed-form expression for the projection onto a capped RSOC. 
        This result facilitates efficient projection-based first-order methods for perspective reformulations. Section~\ref{sec:regress} illustrates its application to the perspective relaxation of sparse logistic regression. The remainder of the paper is organized as follows. We first introduce the mathematical definitions and background, and then present the main results, including a closed-form solution to the projection problem.}


	\subsection{Definitions} 

The Euclidean projection of $w$ onto set $S$, $P_S(w)$, is given by
    \[
    P_S(w)=arg\min_{\ngg{v} \in S} \|\ngg{v} - w\|_2.
    \]
Letting $w_{1:k}$ denote the subvector $(w_1,\ldots,w_{k})$, for the standard second-order cone (SOC) 
    $C=\set{w\in\real^n}{\vectornorm{w_{1:n-1}}\leq w_n}$, the projection is given by the simple closed form~\cite[Theorem 3.3.6]
    {Bauschke96}, 
	\begin{align*}
		P_C(w)=\begin{cases}
			w & \|w_{1:n-1}\|\leq w_n\\
			0 & \|w_{1:n-1}\|\leq -w_n\\
			\frac{\|w_{1:n-1}\|+w_n}{2}\ishyComments{(\frac{w_{1:n-1}}{\|w_{1:n-1}\|},1)} & \text{otherwise.}
		\end{cases}
	\end{align*} 
	Let $C_r$ be the rotated second-order cone (RSOC), defined as $C_r=\set{(w\in\real^{n-2}\times\real^2_{\text{+}}}{\vectornorm{w_{1:n-2}}^2\leq 2w_{n-1}w_{n}}$. Indeed, $C_r$ can be alternatively expressed using the rotation matrix
\begin{align}
			M=\begin{pmatrix}
			I_{n-2} & 0 & 0 \\
			0 & -\frac{1}{\sqrt{2}} & \frac{1}{\sqrt{2}} \\
			0 & \frac{1}{\sqrt{2}} & \frac{1}{\sqrt{2}}
		\end{pmatrix},\nonumber
	\end{align}
 and the SOC $C$, so that $\mathsf{w}\in C_r$ if and only if $M\mathsf{w}\in C$. 
	\ishyComments{Letting $u>0$, the capped RSOC is the intersection of an
RSOC and the half-space $\{\mathsf{w}\in\real^n
\mid w_n\leq u\}$, given by}
			$\mathcal{P} = \left\{\mathsf{w}\in\real^{n} \mid 
            \mathsf{w}\in C_r,\; w_n \leq u\right\}$. 
In the following, we study the problem of evaluating $P_\mathcal P(\hat w)$ for any $\hat w\in\real^n$. 
%
%
In particular, \rev{we} derive and prove the closed-form solution for this problem. 
\ngg{We also make use of the indicator function of a set $\mathbb{I}_S(w)$ given by $\mathbb{I}_S(w)=0$ if $w\in S$ and $\mathbb{I}_S(w)=\infty$ otherwise. The projection onto a closed convex set $S$ is known to equal the proximal operator of the indicator function of $S$, that is, for all $w\in\real^n$, $P_S(w)=\prox_{\mathbb{I}_S}(w)$.} 
\ngg{Hence, our result} is expected to 
    enable projection-based solution methods\ngg{, such as the proximal-gradient method,} for optimization problems whose feasible region typically involves the Cartesian product of a large (finite) number of such sets $\mathcal{P}$. An example of such an application to sparse regression optimization problems will be presented in Section~\ref{sec:regress}. 
	\section{Projecting onto the Capped RSOC}\label{sec:projectsec}
	\rev{We now describe our main result,} 
\ishyComments{starting with a statement of the closed-form expression for
the projection onto the capped RSOC, $\mathcal{P}$.}
	
	\subsection{The Closed Form of the Projection}
	
	\rev{The} projection of $\hat{\mathsf{w}}\in\mathbb{R}^{n}$ onto $\mathcal{P}$ 
	is \rev{given by the}  
    minimizer $\min_{\mathsf{w}\in \mathcal{P}} \{\vectornorm{\mathsf{w}-\hat{\mathsf{w}}}\}$. 
    Consider the following cubic equation
\begin{align}\label{tag18}
\alpha^3+2u(u-\hat y)\alpha-2u^2\|\hat x\|=0,
\end{align}
and let  $\alpha$ denote its nonnegative root, which will be shown to be unique under the condition that $\|\hat{x}\|^2 \geq 2u\hat{y}$. The explicit \ngg{closed-form} expression of $\alpha$ is given 
in Appendix~\ref{app:general_cubic}.
\ishyComments{For convenience, Table~\ref{TableNotation} summarizes the notation used throughout the paper.} 

\begin{table}[t]
\centering
\caption{Summary of notation.}
\label{TableNotation}

\begin{tabular}{ll}
\hline

$C$ 
& Standard second-order cone,
$C=\{w\in\mathbb{R}^n\mid \|w_{1:n-1}\|\leq w_n\}$.\\

$C_r$ 
& Rotated second-order cone,
$C_r=\{w=(x,y,z)\in\mathbb{R}^{n-2}\times\mathbb{R}_+^2
\mid \|w_{1:n-2}\|^2\leq 2w_{n-1}w_n\}$.\\

$M$
& Rotation matrix mapping $C_r$ onto $C$.\\

$u$
& Positive upper bound defining the cap of the rotated second-order cone,
$u>0$.\\

$\mathcal{P}$ 
& Capped rotated second-order cone,
$\mathcal{P}=C_r\cap\{w\in\mathbb{R}^n\mid w_n\leq u\}$,  for $u>0$.\\[0.8ex]

$\widetilde{\mathcal{P}}$
& {Intersection of the boundary of 
$C_r$ with the plane
$\set{w\in\real^n}{w_n=u}$.} \\

$P_S(w)$ 
& Euclidean projection of $w$ onto $S$,
$P_S(w)=\arg\min_{v\in S}\|v-w\|$.\\

$\hat{w}=(\hat{x},\hat{y},\hat{z})$ 
& Point to be projected, where
$\hat{x}\in\mathbb{R}^{n-2}$, $\hat{y},\hat{z}\in\mathbb{R}$.\\

$(x^*,y^*,z^*)$ 
& $P_\mathcal{P}(\hat{x},\hat{y},\hat{z})$, the Euclidean projection onto $\mathcal{P}$.\\


$(\tilde{x},\tilde{y},u)$
& Point in $\widetilde{\mathcal{P}}$, where
$\tilde{y}=\|\tilde{x}\|^2/(2u)$. \\

$\alpha$
& \ishyComments{Unique nonnegative root of \eqref{tag18}, its explicit expression given in Appendix~\ref{app:general_cubic}}.\\[0.8ex]

\hline
\end{tabular}
\end{table}

The following theorem establishes the closed-form expression that solves the projection onto the capped RSOC.
\begin{theorem}\label{opt979}
The projection of
$\hat w=(\hat x,\hat y,\hat z)\in\mathbb R^n$
onto $\mathcal P$ is given by
\[
(x^*,y^*,z^*)
=
\begin{cases}
(\hat x,\hat y,u),
&
\|\hat x\|^2<2u\hat y,\quad u\le\hat z,
\\[1.5ex]
\left(
\dfrac{\alpha}{\|\hat x\|}\hat x,\dfrac{\alpha^2}{2u},u
\right),
&
\hat x\neq 0,\quad \|\hat x\|^2\ge 2u\hat y,\quad
u-\dfrac{\alpha(\|\hat x\|-\alpha)}{2u}\le \hat z,
\\[2ex]
\left(
0,0,u
\right),
&
\hat x=0,\quad \hat y\le 0,\quad
u\le \Hat{z},
\\[2ex]
M P_{\mathcal C}(M\hat w),
&
\text{otherwise}.
\end{cases}
\]
\end{theorem}    

\ngg{\subsection{Key Lemmas and Technical Results}}
	The proof of Theorem \ref{opt979} requires the following lemmas. The first lemma establishes the necessary and sufficient conditions for projection onto the intersection of $\mathcal{P}$ and $\{\mathsf{w}\mid z=u\}$; see also Figure \ref{fig:ex16}.
	
	\begin{lemma}\label{opt994}
		$(\hat{x},\hat{y},\hat{z})\in \mathbb{R}^{n}$ satisfies $\vectornorm{\hat{x}}^2\leq \ishyone{2}u\hat{y}$ and $u\leq \hat{z}$   if and only if $P_\mathcal{P}(\hat{x},\hat{y},\hat{z})=(\hat{x},\hat{y},u)$.
	\end{lemma}
	\begin{proof}
		Assume that $(\hat{x},\hat{y},\hat{z})\in \mathbb{R}^n$ satisfies $P_\mathcal{P}(\hat{x},\hat{y},\hat{z})$=$(\hat{x},\hat{y},u)$. $(\hat{x},\hat{y},u)\in \mathcal{P}$ implies that $\vectornorm{\hat{x}}^2\leq \ishyone{2}u\hat{y}$. By the projection theorem, $(x_{0}-\hat{x},y_{0}-\hat{y},z_{0}-u)\cdot(\hat{x}-\hat{x},\hat{y}-\hat{y},\hat{z}-u)=(z_{0}-u)(\hat{z}-u)\leq 0$ for every  $(x_0,y_0,z_0)\in \mathcal{P}$. In particular, for $(x_0,y_0,z_0)=(\mathbf{0},0,u/2)\in \mathcal{P}$, it follows that
		$-\frac{u}{2}(\hat{z}-u)\leq 0$,
		and since $u>0$, it also follows that $\hat{z}\geq u$. Now, to prove the converse, let $(\hat{x},\hat{y},\hat{z})\in \mathbb{R}^{n}$ satisfy $\vectornorm{\hat{x}}^2\leq\ishyone{2}u\hat{y}$ and $u\leq \hat{z}$, and consider (arbitrary) $(x_{0},y_{0},z_{0})\in \mathcal{P}$. Since $z_{0}\leq u$ and $u\leq\hat{z}$, for $(x,y,z)=(\hat{x},\hat{y},u)$, it follows that $(x_{0}-\hat{x},y_{0}-\hat{y},z_{0}-u)\cdot(\hat{x}-\hat{x},\hat{y}-\hat{y},\hat{z}-u)=(z_{0}-u)(\hat{z}-u)\leq 0$. Therefore, by the  projection theorem, $P_\mathcal{P}(\hat{x},\hat{y},\hat{z})=(\hat{x},\hat{y},u)$.
	\end{proof}
	\begin{figure}	
		\centering
		\includegraphics[scale=0.455]{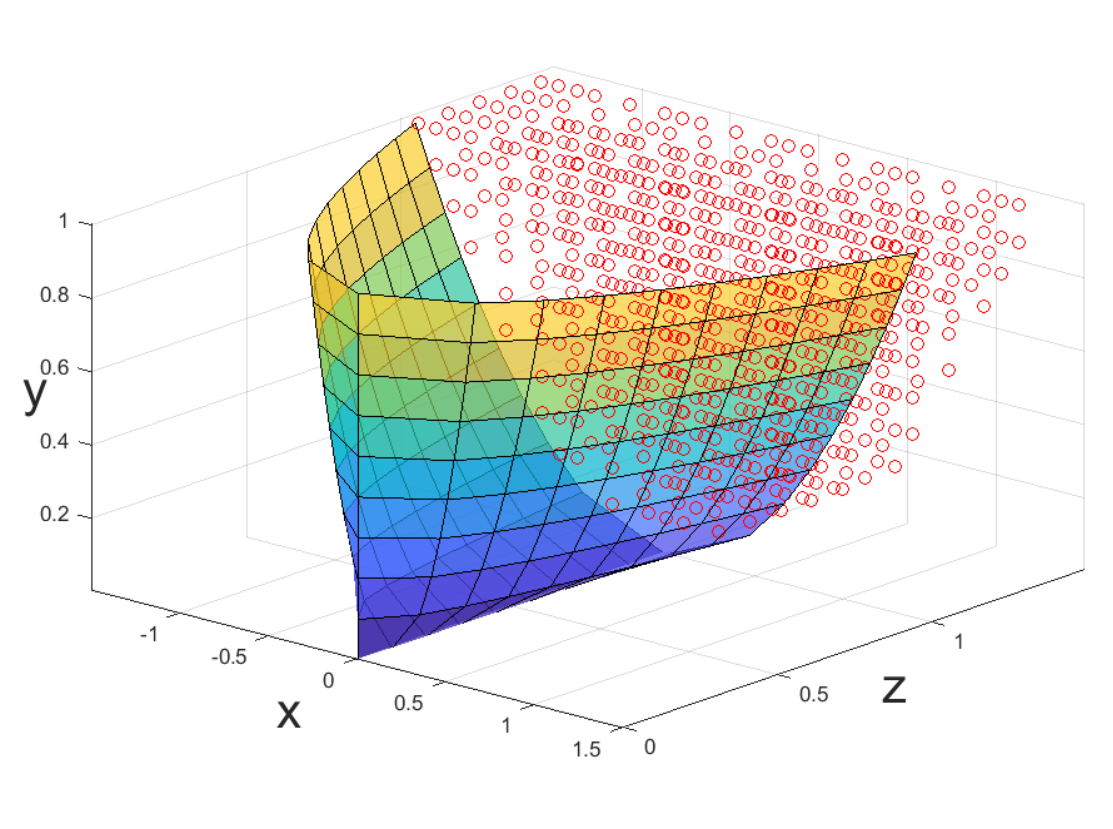}%
		\caption{An illustration of the region that projects onto the plane in the interior of the parabola $\{w=(x,y,z)\in \mathcal{P} \mid z=u \}$ (marked with red circles).}\label{fig:ex16}
        
	\end{figure}
	
	For convenience, in the following, define a subset of the boundary of $\mathcal P$, $
	\widetilde{\mathcal{P}}\equiv\set{(x,y,z)\in \mathcal P} {\vectornorm{x}^2=\ishyone{2}yz,\; z=u}$. This set and the associated regions that project onto it are illustrated in Figure \ref{fig:ex13}. 	
	\begin{figure}
		\centering
		\includegraphics[scale=0.45]{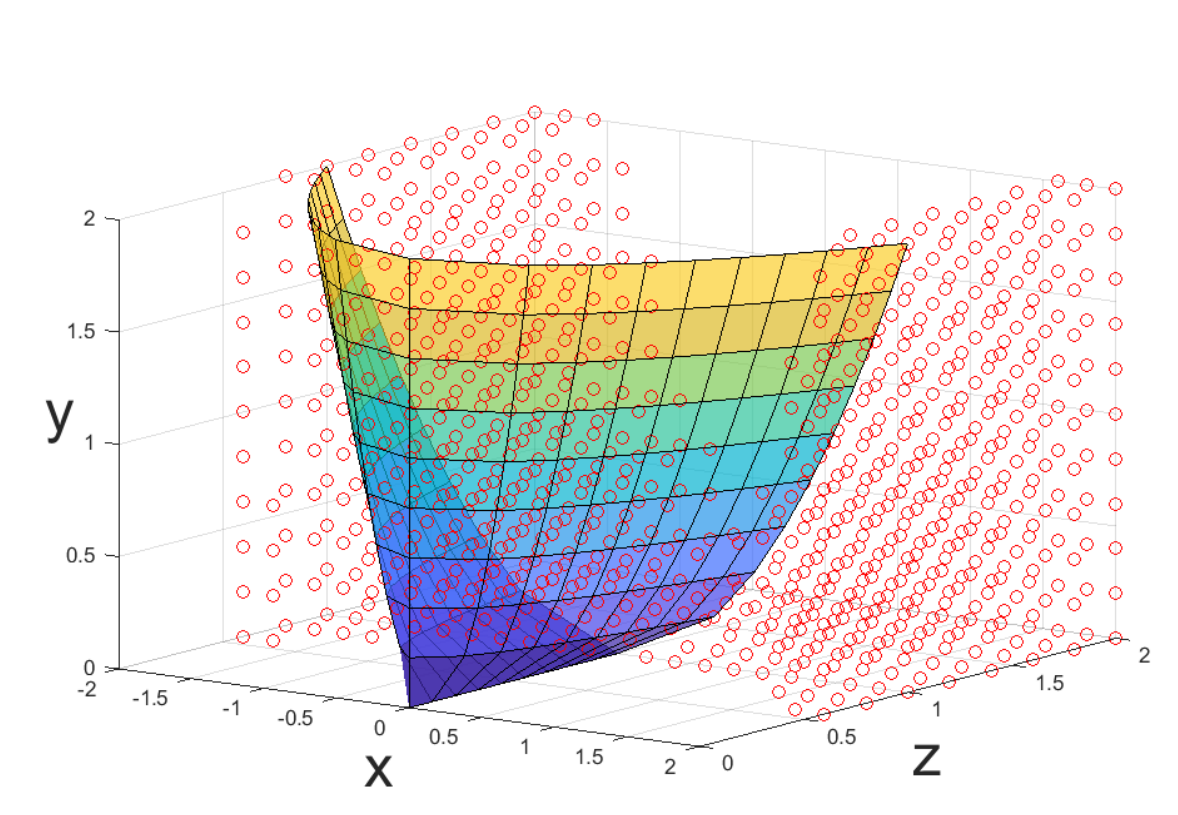}
		\caption{An illustration of the regions (marked with red circles) that get projected onto the parabola in the intersection of the RSOC and the plane (given by $z=u$), $\tilde P=\{(x,y,z)\in \mathcal{P} \mid x^2=2y, \; z=1 \}$.}\label{fig:ex13}
	\end{figure}
	The following lemma establishes necessary and sufficient conditions for projection onto the set $\widetilde{P}$.
    \begin{lemma}\label{ProjectionOnPositiveParabolaPart1}
	Suppose $(x,y,z)\in \widetilde{\mathcal{P}}$ and $(\hat{x},\hat{y},\hat{z})\in \mathbb{R}^{n}$. Then $P_\ishy{\mathcal{P}}(\hat{x},\hat{y},\hat{z})$=$(x,y,z)$ if and only if 
	\begin{enumerate}[label={[\roman*]},wide=0pt]
		\item\label{item5} $\hat{y}\leq \frac{\vectornorm{x}^2}{\ishy{2u}}$, 
        \item\label{item6} $\vectornorm{x} \leq \vectornorm{\hat{x}}$, \item\label{item7} $\hat{z}\geq u-\frac{1}{u}x\cdot(\hat{x}-x)-\frac{y}{u}(\hat{y}-\frac{\vectornorm{x}^2}{\ishy{2u}})$ and 
        \item\label{item8} $x=\frac{\ishy{2}u^2}{\vectornorm{x}^2-2u\hat{y}+\ishy{2}u^2}\hat{x}$.
	\end{enumerate}
\end{lemma}
\begin{proof}
	Suppose that $(x,y,z)\in \widetilde{\mathcal{P}}$ and $(\hat{x},\hat{y},\hat{z})\in \mathbb{R}^{n}$, such that they \ishy{satisfy \ref{item5}-\ref{item8}}, and consider arbitrary $(x_{0},y_{0},z_{0})\in \mathcal{P}$. First, suppose that $\hat{y}=y$. Then, \ishy{ condition \ref{item8}} implies that $\hat{x}=x$ while condition \ref{item7} implies $\hat{z}\geq u$. Then, it follows that 
    \begin{align*}
		(x_{0}-x,y_{0}-y,z_{0}-z)\cdot&(\hat{x}-x,\hat{y}-y,\hat{z}-z)   \\  &=(z_{0}-z)(\hat{z}-z) \leq 0.
	\end{align*}
	
	\noindent Otherwise $\hat{y}<y$. Now consider the case where $z_0=0$. Thus, $\|x_0\|^2\leq \ishy{2}y_0 z_0 =0$, implies that $x_0=\mathbf{0}$. Then, together with $\ref{item7}, \ref{item5}$ and the fact that $y_0\geq 0$, it follows that 
	
	\begin{align}
		& && (x_{0}-x,y_{0}-y,z_{0}-z)\cdot(\hat{x}-x,\hat{y}-y,\hat{z}-z) \\& =  && -(x\cdot(\hat{x}-x)+y(\hat{y}-y)+u(\hat{z}-u))+y_{0}(\hat{y}-y) \nonumber\\
		& \leq && -u(\hat{z}-( u-\frac{1}{u}x\cdot(\hat{x}-x)-\frac{y}{u}(\hat{y}-y))) \leq 0.\nonumber
	\end{align}
	Otherwise $z_0\neq 0$. Then, since $\frac{\|x_0\|^2}{\ishy{2}z_0}\leq y_0$, it follows from \ref{item5} that $\hat{y}-y<0$ and therefore
	\begin{align}	
		& && (x_0-x,y_0-y,z_0-u)\cdot(\hat{x}-x,\hat{y}-y,\hat{z}-u)\\ &\leq&&(x_0-x,\frac{\|x_0\|^2}{\ishy{2}z_0}-y,z_0-u)\cdot(\hat{x}-x,\hat{y}-y,\hat{z}-u)\\ & =&& \Big(\frac{\hat{y}-y}{\ishy{2}z_0}\Big)\|x_0\|^2+x_0\cdot(\hat{x}-x)-x\cdot(\hat{x}-x)\nonumber\\
		& && -y(\hat{y}-y)+(z_0-u)(\hat{z}-u)\nonumber\\
		& \leq &&\Big(\frac{\hat{y}-y}{\ishy{2}z_0}\Big)\|x_0\|^2+\|(\hat{x}-x)\|\| x_0\|-x\cdot(\hat{x}-x)\nonumber\\ & &&-y(\hat{y}-y)+(z_0-u)(\hat{z}-u).\nonumber
	\end{align}
	The last inequality followed from the Cauchy-Schwarz inequality. (In particular, the equality holds throughout for \ishyComments{$\|x_0\|^2=2y_0z_0$} and \ishyComments{$x_0=\beta(\hat{x}-x)$} for \ishyComments{arbitrary $\beta>0$}.) This quadratic polynomial in $\|x_0\|$ is nonpositive, if and only if the quadratic coefficient $ \frac{\hat{y}-y}{\ishyComments{2}z_0}  <0 $ (which is satisfied in this case) and the discriminant $D\leq 0$, where	
	\begin{align}\label{opt819}
		D & = &&\|\hat{x}-x\|^2-\ishy{2}\Big( \frac{\hat{y}-y}{z_0} \Big)\\
        & && \Big(-x\cdot(\hat{x}-x)  -y(\hat{y}-y) +(z_0-u)(\hat{z}-u)\Big)\nonumber\\ 
		& = &&\|\hat{x}-x\|^2-\ishy{2}(\hat{y}-y)(\hat{z}-u)\nonumber\\
        & && +\ishy{2}(\hat{y}-y)(x\cdot(\hat{x}-x) +y(\hat{y}-y) +u(\hat{z}-u))(\frac{1}{z_0})\nonumber\\
		& \leq && \|\hat{x}-x\|^2-\ishy{2}(\hat{y}-y)(\hat{z}-u)\nonumber\\
        & && +\ishy{2}(\hat{y}-y)(x\cdot(\hat{x}-x)  +y(\hat{y}-y)+u(\hat{z}-u))(\frac{1}{u})\nonumber\\
		& = && \|\hat{x}-x\|^2+\ishy{2}(\hat{y}-y)(x\cdot(\hat{x}-x)+y(\hat{y}-y))(\frac{1}{u})\nonumber\\
		&  \equiv && \bar{D}.\nonumber
	\end{align}
	(In particular, $D=\bar{D}$ for $z_0=u$.) The last inequality followed from \ref{item5}, \ref{item7} and since $0\leq z_{0} \leq u$.\\
	Evidently, $\bar{D}=\frac{\|\sqrt{u}(\hat{x}-x)+\frac{\hat{y}-y}{\sqrt{u}}x\|^2_2}{u}\geq 0$
	follows from $u>0$ and $\vectornorm{x}^2_{2}=\ishy{2}uy$.
	So,
	\begin{align}\label{Formu9918}
		\bar{D}=0\iff&\sqrt{u}(\hat{x}-x)+\frac{\hat{y}-y}{\sqrt{u}}x  = \mathbf{0}\\
		\overset{\substack{\text{substituting }\\ y=\vectornorm{x}^2/\ishy{2}u}}{\iff}& x=\frac{\ishy{2}u^2}{\vectornorm{x}^2-2u\hat{y}+\ishy{2}u^2}\hat{x}.\nonumber
	\end{align}
	\ishyComments{Thus, \ref{item8} implies} that $0=\bar{D}\geq D$. 
	It follows that $(\hat{x}-x,\hat{y}-y,\hat{z}-z)\cdot(x_0-x,y_{0}-y,z_0-z)\leq 0$ for all $(x_0,y_{0},z_0)\in \mathcal{P}$, and by the projection theorem, $P_\mathcal{P}(\hat{x},\hat{y},\hat{z})$=$(x,y,z)$.
	
	Now, to prove the converse, assume that $(\hat{x},\hat{y},\hat{z})\in \mathbb{R}^{n}$ satisfies $P_\mathcal{P}(\hat{x},\hat{y},\hat{z})$=$(x,y,z)$ where $(x,y,z)\in \widetilde{\mathcal{P}}$. By the projection theorem, $(x_0-x,y_0-y,z_0-z)\cdot(\hat{x}-x,\hat{y}-y,\hat{z}-z)\leq 0$ for every $(x_0,y_0,z_0)\in \mathcal{P}$. In particular, for $(x_0,y_0,z_0)=(\mathbf{0},0,0)\in \mathcal{P}$, it follows that
	\begin{align*}
	x\cdot(\hat{x}-x)+y(\hat{y}-y)+u(\hat{z}-u) & =  u(\hat{z}- (u-\frac{1}{u}x\cdot(\hat{x}-x)-\frac{y}{u}(\hat{y}-y)))\\ &\geq 0,
	\end{align*}
	and thus \ref{item7} holds. For $(x_0,y_0,z_0)=(0, \frac{\vectornorm{x}^2}{\ishy{2u}},u)\in \mathcal{P}$, it follows that
	\begin{align*}
		0 &\geq && (x_0-x,y_0-y,z_0-u)\cdot(\hat{x}-x,\hat{y}-y,\hat{z}-u) \\ 
        &=&&-x\cdot(\hat{x}-x)=\vectornorm{x}^2-x\cdot\hat{x}\geq\vectornorm{x}^2-\vectornorm{x}\cdot\vectornorm{\hat{x}}.
	\end{align*}
	The last inequality followed from Cauchy-Schwarz. Hence \ref{item6} holds. Now assume, for the sake of contradiction, that $\hat{y}>y$. \ishyComments{Then, from $(x,y,z)\in \widetilde{\mathcal{P}}$, it follows that $(x,\hat{y},u)\in \mathcal{P}$}, and by choosing  $(x_0,y_0,z_0)=(x,\hat{y},u)\in \mathcal{P}$ it follows that
	\begin{align*}
		(x_0-x,y_0-y,z_0-u)\cdot(\hat{x}-x,\hat{y}-y,\hat{z}-u) &=(\hat{y}-y)^2  >0,
	\end{align*}
	a contradiction to the projection theorem. Thus $\hat{y}\leq y=\frac{\vectornorm{x}^2}{\ishy{2u}}$ and \ref{item5} holds.
	From \ref{item5}, \ref{item6}, \ref{item7}, \eqref{opt819}, and the fact that $(x_0-x,y_0-y,z_0-z)\cdot(\hat{x}-x,\hat{y}-y,\hat{z}-z)\leq 0$ for every $(x_0,y_0,z_0)\in \mathcal{P}$, in particular for \ishyComments{$x_0=\beta(\hat{x}-x)$, $y_0=\frac{\beta^2\|\hat{x}-x\|^2}{\ishy{2}u}$} and $z_{0}=u$, for  \ishyComments{arbitrary $\beta>0$}, it follows that $\bar{D}=0$. By \eqref{Formu9918}, \ishyComments{it follows that \ref{item8} holds}.\unskip \end{proof}
\ishyComments{The conditions \ngg{with respect to the projection solution $(x,y,z)\in \widetilde{\mathcal P}$} in Lemma~\ref{ProjectionOnPositiveParabolaPart1} arise naturally from the projection theorem\ishyComments{,}
\ngg{providing a useful} characterization \ngg{of this  solution}. The
following lemma shows that these conditions are equivalent to a more
compact set of conditions \ngg{in terms of the problem input, that is\ishyComments{,} the point to be projected $(\hat x,\hat y, \hat z)$}, 
as used in the statement of the
closed-form projection formula in Theorem~\ref{opt979}.}

\begin{lemma}\label{lem:conditions-equivalence}
Let $(\hat{x},\hat{y},\hat{z})\in\mathbb{R}^n$.

\begin{enumerate}[label={(\alph*)},wide=0pt]

    \item\label{lem:conditions-equivalence:a}
    If
    \begin{equation}
        \|\hat{x}\|^2 \geq 2u\hat{y},
        \label{cond:A}\tag{A}
    \end{equation}
    then the cubic equation \eqref{tag18} has a unique nonnegative
    root $\alpha$, given in
    Appendix~\ref{app:general_cubic}. Moreover,
    \[
        0\leq\alpha\leq\|\hat{x}\|.
    \]

    \item\label{lem:conditions-equivalence:b}
    Let $(\tilde{x},\tilde{y},u)\in\widetilde P$, and let $\alpha$
    be the unique nonnegative root of \eqref{tag18}.
    The conditions
    \begin{enumerate}[label={[\roman*]},wide=0pt]
        \item\label{item1}
        $2u\hat{y}\leq \vectornorm{\tilde{x}}^2$,

        \item\label{item2}
        $\vectornorm{\tilde{x}}\leq\vectornorm{\hat{x}}$,

        \item\label{item3}
        $u\hat{z}\geq
        u^2-\tilde{x}\cdot(\hat{x}-\tilde{x})
        -\tilde{y}(\hat{y}-\tilde{y})$, and

        \item\label{item4}
        $\tilde{x}
        =\frac{2u^2}
        {\vectornorm{\tilde{x}}^2-2u\hat{y}+2u^2}\hat{x}$.
    \end{enumerate}
    hold if and only if
    \[
    \tilde{x}=
    \begin{cases}
        \dfrac{\alpha}{\|\hat{x}\|}\hat{x},
        & \|\hat{x}\|>0,\\[1ex]
        0,
        & \|\hat{x}\|=0,
    \end{cases}
    \]
    and condition~\eqref{cond:A} holds together with
    \begin{equation}
        \hat{z}\geq
        u-\frac{\alpha(\|\hat{x}\|-\alpha)}{2u}.
        \label{cond:B}\tag{B}
    \end{equation}

\end{enumerate}
\end{lemma}
\begin{proof}~
\paragraph{Proof of \ref{lem:conditions-equivalence}\ngg{-}\ref{lem:conditions-equivalence:a}.} 
Define
\[
g(t)
=
t^3+2u(u-\hat{y})t-2u^2\|\hat{x}\|.
\]
We first show that $g$ has a unique nonnegative root. \ishyComments{If $\|\hat{x}\|=0$, then condition~\eqref{cond:A} implies that
$\hat{y}\leq 0$. In this case,
\[
g(t)
=
t\left(t^2+2u(u-\hat{y})\right).
\]
Since $u>0$ and $\hat{y}\leq0$, $t=0$ must be the unique nonnegative root of
$g$. Otherwise, 
we have
\[
g(0)=-2u^2\|\hat{x}\|<0
\qquad\text{and}\qquad
\lim_{t\to\infty}g(t)=\infty,
\] so $g$ has at least one positive root.}

To prove uniqueness, consider the derivative of $g$,
\[
g'(t)=3t^2+2u(u-\hat{y}).
\]
\ishyComments{First, if} $\hat{y}\leq u$, then $g'(t)>0$ for every $t>0$, so $g$ is
strictly increasing on $(0,\infty)$ and therefore has exactly one
positive root. \ishyComments{Otherwise, if $\hat{y}>u$, then $g'$ has exactly one positive zero \ngg{at} $t_0=\sqrt{2u(\hat{y}-u)/3}.$
Hence, $g$ decreas\ngg{es} on $(0,t_0)$ and increas\ngg{es} on
$(t_0,\infty)$. Since $g(0)<0$, we also have $g(t_0)<0$.
Therefore, $g$ 
\ngg{has exactly one positive root, specifically on $(t_0,\infty)$}.} Thus, in either case, $g$ has a unique positive root. Denote it by
$\alpha$. Moreover, by condition~(A),
\[
g(\|\hat{x}\|)
=
\|\hat{x}\|
\left(\|\hat{x}\|^2-2u\hat{y}\right)
\geq0,
\]
implying that
\[
0<\alpha\leq\|\hat{x}\|.
\]

It remains to identify this root with \eqref{eq:alpha_closed_form} in
Appendix~\ref{app:general_cubic}. Equation~\eqref{tag18} is a depressed cubic with
\[
p=2u^2-2u\hat{y},
\qquad
q=-2u^2\|\hat{x}\|,
\]
and Appendix~\ref{app:general_cubic} classifies its roots according to the sign of the
discriminant 
\[\Delta
=\frac{-u^3(8u^3-24\hat{y}u^2+(24\hat{y}^2+27\|\hat{x}\|^2)u-8\hat{y}^3)}{27}.\]
If $\Delta<0$, the cubic has a unique real root, given by \eqref{eq:alpha_closed_form}. Since $\alpha$ is the unique nonnegative root, this real
root is $\alpha$.

If $\Delta=0$, the cubic has three real roots (see Appendix~\ref{app:general_cubic})
\[
\alpha_1=2\sqrt[3]{u^2\|\hat{x}\|},
\qquad
\alpha_2=\alpha_3=-\sqrt[3]{u^2\|\hat{x}\|}.
\]
Since $\|\hat{x}\|>0$ and $u>0$, we have
\[
\alpha_1>0,
\qquad
\alpha_2=\alpha_3<0.
\]
Therefore, $\alpha=2\sqrt[3]{u^2\|\hat{x}\|}$ must be the unique nonnegative root of the cubic.\\
If $\Delta>0$, we have three distinct real roots. Since
$\|\hat{x}\|>0$, we have
\[
q=-2u^2\|\hat{x}\|<0.
\]
By Vieta's formulas, $\alpha_1\alpha_2\alpha_3=-q>0$ and $\alpha_1+\alpha_2+\alpha_3=0$. Consequently, exactly one of $\alpha_1,\alpha_2,\alpha_3$ is positive
and the other two are negative. Hence, the positive root given in
Appendix~\ref{app:general_cubic} is precisely $\alpha$. 

\paragraph{Proof of~\ref{lem:conditions-equivalence}-\ref{lem:conditions-equivalence:b}.} 
We first show that conditions~\ref{item1}--\ref{item4} imply conditions~\eqref{cond:A}--\eqref{cond:B}.

By conditions~\ref{item1} and~\ref{item2},
\[
2u\hat y
\leq
\|\tilde{x}\|^2
\leq
\|\hat{x}\|^2,
\]
and hence $2u\hat y\leq\|\hat{x}\|^2$ which is condition~\eqref{cond:A}.

By condition~\ref{item1}, $\|\tilde{x}\|^2-2u\hat{y}+2u^2\geq2u^2>0$. This implies that condition~\ref{item4} is defined, and by taking norms we have
\[
\|\tilde{x}\|
=
\frac{2u^2}
{\|\tilde{x}\|^2-2u\hat{y}+2u^2}\|\hat{x}\|.
\]
Equivalently, $\|\tilde{x}\|^3+2u(u-\hat{y})\|\tilde{x}\|-2u^2\|\hat{x}\|=0$.
Since $\|\tilde{x}\|\geq0$ and $\alpha$ is the unique nonnegative root of this
cubic equation, it follows that $\|\tilde{x}\|=\alpha$.
Now, if $\|\hat{x}\|=0$, by condition~\ref{item4}
$\tilde{x}=0$. If $\|\hat{x}\|>0$, substituting
$\|\tilde{x}\|=\alpha$ yields
\[
\tilde{x}
=
\frac{\alpha}{\|\hat{x}\|}\hat{x}.
\]
By condition~\ref{item3}, the fact that $\tilde{y}=\vectornorm{\tilde{x}}^2/2u=\alpha^2/2u$, and that $\tilde{x}$ is a scalar multiple of $\hat{x}$ 
\ngg{so} that $\tilde{x}\cdot\hat{x}=\vectornorm{\tilde{x}}\vectornorm{\hat{x}}=\alpha\vectornorm{\hat{x}}$, \ishyComments{and $g(\alpha)=0$, together} \ngg{imply that}
\begin{align}\label{tag89}
u\hat{z}
\geq
u^2-\tilde{x}^{\top}(\hat{x}-\tilde{x})
-\tilde{y}(\hat{y}-\tilde{y})=
u^2-\alpha(\|\hat{x}\|-\alpha)
+\frac{\alpha(\|\hat{x}\|-\alpha)}{2}=
u^2-\frac{\alpha(\|\hat{x}\|-\alpha)}{2},
\end{align}
which is condition~\eqref{cond:B}.

We now prove the converse. Suppose that conditions~\eqref{cond:A}
and~\eqref{cond:B} hold, and that $\tilde{x}$ and $\tilde{y}$ are
as given in the statement of the lemma. 
First note that given condition~\eqref{cond:B} the equality in \eqref{tag89} holds as $\tilde{x}$ is again defined as a scalar multiple of $\hat{x}$, so \ref{item3} holds.
By part \ref{lem:conditions-equivalence:a},
$\alpha\in[0,\|\hat{x}\|]$. Since $\|\tilde{x}\|=\alpha$, condition~\ref{item2} follows. Moreover, since $g(\alpha)=0$, $2u^2(\|\hat{x}\|-\alpha)
=
\alpha
\left(
\alpha^2-2u\hat{y}
\right)$.
\ishyComments{If $\alpha=0$, then $\hat{x}=0$, and condition~\eqref{cond:A} gives
$2u\hat{y}\leq 0=\alpha^2$.
If $\alpha>0$, the nonnegativity of the left-hand side implies
$\alpha^2\geq 2u\hat{y}$.
Thus, in either case,
$\|\tilde{x}\|^2=\alpha^2\geq 2u\hat{y}$,
which is condition~\ref{item1}.}
Since $g(\alpha)=0$, we have
$\alpha
\left(
\alpha^2-2u\hat{y}+2u^2
\right)
=
2u^2\|\hat{x}\|$. Now, if $\vectornorm{\hat{x}}>0$, then $\alpha\geq 0$ implies that
\[
\frac{\alpha}{\|\hat{x}\|}
=
\frac{2u^2}
{\alpha^2-2u\hat{y}+2u^2}.
\]
So,
\[
\tilde{x}
=
\frac{\alpha}{\|\hat{x}\|}\hat{x}
\qquad\text{and} \qquad
\tilde{x}
=
\frac{2u^2}
{\|\tilde{x}\|^2-2u\hat{y}+2u^2}
\hat{x},
\]
which is condition~\ref{item4}.
If $\hat{x}=0$, then $\tilde{x}=0$ in which case condition~\ref{item4} also holds. Thus, conditions~\ref{item1}--\ref{item4} hold.
\end{proof}


The following lemma facilitates efficient projection onto an 
RSOC, by utilizing the known projection onto the standard SOC.
\begin{lemma}\label{projectionOntoTheBoundOfScaledRSOC}
	$P_{C_r}(\mathsf{w})=MP_{C}(M\mathsf{w})$
\end{lemma}
\begin{proof}
	Let $\mathcal{A} \in \mathbb{R}^{m \times n}$, $\mathcal{B} \in \mathbb{R}^{m}$, and $\lambda>0$.
	 According to Theorem 6.15 in \cite
     {Beck17}, if $\mathcal{A}\mathcal{A}^T=(1/\lambda)I$ and $f(\mathsf{w})=g(\mathcal{A}\mathsf{w}+\mathcal{B})$ then $prox_f(\mathsf{w})=(I-\lambda \mathcal{A}^T\mathcal{A})\mathsf{w}+\lambda\mathcal{A}^T(prox_{\lambda^{-1}}g(\mathcal{A}\mathsf{w}+\mathcal{B})-\mathcal{B})$.
	In particular, for $\lambda=1$, $\mathcal{B}=\textbf{0}$, $\mathcal{A}=M$, $f(\mathsf{w})=\mathbb{I}_{C_r}(\mathsf{w})$ and $g(\mathsf{w})=\mathbb{I}_C(\mathsf{w})$. It follows that $P_{C_r}(\mathsf{w})=prox_f(\mathsf{w})=Mprox_g(M\mathsf{w})=MP_{C}(M\mathsf{w})$.
\end{proof}
	
    \ngg{\subsection*{Proof of Theorem \ref{opt979}} }
	\begin{proof}

    We first show that the \ngg{first three} cases of the expression in the hypothesis hold using Lemmas~\ref{opt994}--\ref{lem:conditions-equivalence}.
		
		\begin{Case}\label{Case1} \underline{$\vectornorm{\hat{x}}^2< {2}u\hat{y}$, $u\leq\hat{z}$ (``projection onto the plane \ngg{$\set{(x,y,z)\in\real^n}{z=u}$}"; see Figure~\ref{fig:ex16}):}
		\end{Case}
		By Lemma \ref{opt994}, $(x^{*},y^{*},z^{*})=(\hat{x},\hat{y},u)$. \ngg{The current case   excludes only the boundary given by $\|\hat x\|^2 = 2u\hat y$, $u\le\hat z$,
of the region in Lemma~\ref{opt994}; there $\alpha=\|\hat x\|$ (or $\hat x=0$), and Case~\ref{Case2} (or Case~\ref{Case3} if $\hat{x}=0$) returns the same point $(\hat x,\hat y,u)$.}
		
		\begin{Case}\label{Case2}  \underline{$\hat{x}\neq 0$, $2u\hat{y}\leq\|\hat{x}\|^2$, $u-\frac{\alpha(\|\hat{x}\|-\alpha)}{2u}\leq\hat{z}$ (``projection onto the parabola $\widetilde{\mathcal{P}}$"; see Figure~\ref{fig:ex13}):}\end{Case}
		By Lemma \ref{ProjectionOnPositiveParabolaPart1}, the Euclidean projection of $(\hat x,\hat y,\hat z)$ onto $P$ is given by $(x^{*},y^{*},z^{*})=\left(
\frac{\alpha}{\|\hat{x}\|}\hat{x},
\frac{\alpha^2}{2u},
u
\right)$. Note that the current case and Case \ref{Case1} are disjoint 
        since $2u\hat{y}\leq\vectornorm{\hat{x}}^2$.

\ishyComments{\begin{Case}\label{Case3}\underline{$\hat{x}=0$,  $\hat{y}\leq0$, $u\leq\hat{z}$ \ngg{(``projection onto the vertex $(0,0,u)$")}:}\end{Case}
Since $\hat{x}=0$ and $\hat{y}\leq0$, we have $2u\hat{y}\leq 0=\|\hat{x}\|^2$ so condition~\eqref{cond:A} of Lemma~\ref{lem:conditions-equivalence}\ref{lem:conditions-equivalence:a} holds. By Lemma~\ref{lem:conditions-equivalence}\ref{lem:conditions-equivalence:a}, the unique
nonnegative root of \eqref{tag18} is therefore
$\alpha=0$. Consequently, condition~\eqref{cond:B} of Lemma~\ref{lem:conditions-equivalence}\ref{lem:conditions-equivalence:b} reduces to $u\leq\hat{z}$, which holds by assumption. Hence, by Lemma~\ref{ProjectionOnPositiveParabolaPart1}, $(x^{*},y^{*},z^{*})=(0,0,u)$. This case is disjoint from Case~\ref{Case1} since $\hat{x}=0$ and
$\hat{y}\leq0$ imply $2u\hat{y}\leq 0=\|\hat{x}\|^2$ whereas Case~\ref{Case1} requires $\|\hat{x}\|^2<2u\hat{y}$.
It is also disjoint from Case~\ref{Case2}, since Case~\ref{Case2} requires \ngg{that}
$\hat{x}\neq0$.}

\begin{Case}\label{Case4} \underline{Otherwise (``projection onto the uncapped RSOC”):}\end{Case}
        \ngg{First, we show that the negation of the (necessary and
		sufficient) conditions of the preceding cases \ngg{(Cases~\ref{Case1}-\ref{Case3})} implies that $z^{*}<u$ in the following complementary case ``otherwise'', hence the projection in this case amounts to projecting onto the standard rotated SOC.} 
        Suppose for the sake of deriving a contradiction that \ishyComments{none of Cases~\ref{Case1}--\ref{Case3} hold} and that $z^{*}=u$. 
        \ngg{We proceed by establishing a contradiction in }   
        each of the following \ngg{two} cases: 
\begin{enumerate}[label={[\roman*]}] \item\label{Case3.1}\underline{$\vectornorm{x^{*}}^2=2y^{*}z^{*}=2y^{*}u$:}\\
			\ishyComments{In this case, $(x^*,y^*,z^*)\in\widetilde P$. Hence, by Lemmas~\ref{ProjectionOnPositiveParabolaPart1} and~\ref{lem:conditions-equivalence}\ref{lem:conditions-equivalence:b}, conditions~\eqref{cond:A}--\eqref{cond:B} hold.
If $\hat{x}\neq0$, then the conditions of Case~\ref{Case2} hold, which is a
contradiction.

If $\hat{x}=0$, condition~\eqref{cond:A} implies $\hat{y}\leq0$. Moreover, by
Lemma~\ref{lem:conditions-equivalence}-\ref{lem:conditions-equivalence:a}, $\alpha=0$, so condition~\eqref{cond:B} reduces to $\hat{z}\geq u$.
Thus, the conditions of Case~\ref{Case3} hold, again yielding a contradiction.}
			 \item\label{Case3.2}\underline{$\vectornorm{x^{*}}^2<2y^{*}z^{*}=2y^{*}u$:}\\
             \ishyComments{Since only the constraint $z\leq u$ is active at
$(x^*,y^*,z^*)$, the projection agrees with the
projection onto the half-space
$\{(x,y,z)\in\mathbb{R}^n\mid z\leq u\}$.
Therefore, $(x^*,y^*,z^*)=(\hat{x},\hat{y},u)$, and} by Lemma \ref{opt994}, $\vectornorm{\hat{x}}^2\leq \ishyone{2}u\hat{y}$ and $u\leq \hat{z}$. 
			 Since $\vectornorm{x^{*}}^2<2y^{*}z^{*}$ and $(x^{*},y^{*},z^{*})=(\hat{x},\hat{y},u)$, it follows that $\vectornorm{\hat{x}}^2 < 2u\hat{y}$, and together with $u\leq \hat{z}$, it follows that Case \ref{Case1} holds, thereby establishing a contradiction.
		\end{enumerate}
Thus, from the negation of the necessary and sufficient conditions
of Cases~\ref{Case1}--\ref{Case3}, it follows that $z^*<u$. Hence, the constraint
$z\leq u$ is inactive at the unique optimal solution.
Therefore, \ngg{the constraint $z\leq u$ is redundant and }
by Lemma~\ref{projectionOntoTheBoundOfScaledRSOC} \ngg{it follows that}
\[
P_\mathcal{P}(\hat{w})=M P_C(M\hat{w}).\qedhere\]
	\end{proof}
	

    \section{An Application to Sparse Logistic Regression}\label{sec:regress}    \ishyComments{In this section, we demonstrate how the projection formula derived in
Section~\ref{sec:projectsec} can be used to solve the continuous perspective relaxation
of sparse logistic regression.} Given a data matrix $A\in\mathbb{R}^{t\times n}$ and a binary response vector
$b\in\{0,1\}^{t}$, let $A_i\in\mathbb{R}^{1\times n}$ denote the $i$-th row of
$A$, \ngg{for} $i=1,\ldots,t$. \ngg{Logistic regression involves determining the best-fitting $x$ by minimizing the loss function $L(x)=-\sum_{i=1}^{t}
\left(
b_i\log(p_i(x))
+(1-b_i)\log(1-p_i(x))
\right)$, where $p_i(x)=(1+e^{-A_i x})^{-1}$.}  
\ngg{Motivated by} 
\ishyComments{
the sparse logistic regression problem with both \(L_0\) and \(L_2\) regularization (see~\cite{Pilanci2015}; see also~\cite{bertsimas17logistic}):}
\begin{align}
\min_{x\in\mathbb{R}^n}
\quad
&
L(x)
+\gamma_0\|x\|_0
+\gamma_2\|x\|_2^2,
\label{eq:SparseLogisticOriginal}
\end{align}
\ngg{we consider its} continuous perspective relaxation 
\begin{align*}
	&	\min_{x, y, z} && 
    L(x) + \gamma_0 \sum_{j=1}^n z_j + \gamma_2 \sum_{j=1}^n y_j \\
	&	\text{s.t.} &&  
x_j^2\leq y_jz_j &&  j \in [n] \\
		&  && 0 \leq y_j, &&  j \in [n] \\
		& && 0 \leq z_j \leq 1 &&   j \in [n].
	\end{align*}

Substituting $z_j = 2\bar z_j$ into the objective 
and defining $\bar\gamma_0=2\gamma_0$, 
for $u=0.5$, 
this problem can be rewritten as 
\begin{subequations}\label{form:nonlinlogreg}\begin{align}
& \min_{x,y,\bar z} && 
L(x)
+\bar{\gamma}_0\sum_{j=1}^{n}\bar z_j
+\gamma_2\sum_{j=1}^{n}y_j
\label{form:logisticUsingOurObject}
\\
&\text{s.t.} &&
(x_j,y_j,\bar z_j)\in\mathcal P, && 
j\in [n].
\end{align}\end{subequations}

\ngg{The formulation~\eqref{form:nonlinlogreg} has a (convex) differentiable nonlinear objective function subject to capped rotated second\ishyComments{-}order cone constraints ($(x_j,y_j,\bar z_j)\in \mathcal P$). This problem can be cast as an unconstrained problem that is the sum of the smooth convex function and a nonsmooth yet proximable indicator function (whose proximal point operator is the projection onto $\mathcal P$). 
Consequently, Theorem~\ref{opt979} provides 
this projection operator that is required by 
projection-based methods (or more generally proximal point methods)
such as the FISTA algorithm~\cite{Beck09}
.} FISTA is an accelerated gradient method for solving composite convex
optimization problems consisting of a smooth convex function with a
Lipschitz-continuous gradient and a nonsmooth convex function with an
efficiently computable proximal operator. \ngg{For completeness, pseudocode of the algorithm is given in Appendix~\ref{app:fista}.}

This formulation~\eqref{form:nonlinlogreg} 
can be further rewritten 
as a linear conic optimization problem, so that it would be compatible with conic optimization solvers. 
The linear conic formulation with second-order cone and exponential cone constraints is provided in Appendix~\ref{sec:LogRegForSCS} as implemented for the SCS solver~\cite{ODonoghue_Boyd_2016}. \ishyComments{In the following computational experiments, we compare FISTA and
SCS for solving the continuous perspective relaxation of the sparse
logistic regression problem. The FISTA method was implemented using
the FOM MATLAB toolbox~\cite{beck2019fom}, while the conic formulation was
solved using SCS~\cite{ODonoghue_Boyd_2016}.} We implemented the models and algorithms using MATLAB version 9.13 and SCS version 3.2.8.

For SCS, we used the direct solver with default parameters, and for FISTA,
we used the FOM MATLAB toolbox, \ngg{also configured with the default settings other than} increasing the default maximum number of
iterations (to $100{,}000$). In particular, we used the default stopping criterion, under which the method terminates when 
the norm of the difference between consecutive iterates \ishyComments{falls below the default tolerance $\epsilon=10^{-5}$. Step sizes are chosen by the toolbox's default backtracking line search.} The experiments were conducted on a server with 768GB of memory and a 4.5 GHz CPU with 96 cores, 192 logical processors and a 384 MB cache size. The 
logistic regression problem 
datasets were generated using \ngg{the} \texttt{make\_classification} \ngg{procedure available} from Scikit-learn~\cite{scikit-learn}. We considered datasets with numbers of observations $t=100$ and numbers of features $n\in\{2000,4000,\dots,14000,20000\}$. This data generation process is based on a ground-truth binary classification model with $\ngg{\hat n=}\lfloor 0.1n\rfloor$ informative features. We used penalty parameters $\gamma_2=0.5$ and $\gamma_0=0.1t$. All experiments were conducted using a single CPU core, and hence CPU
time is used as the running-time measure. Table~\ref{table:LogisticRegression1TimeTenthInterest} reports the performance of the solvers SCS and FISTA in terms of CPU time and number of iterations (It.) for solving the sparse logistic regression problem.

\begin{table}[ht!]
		\centering
		\begin{tabular}{
				|r|rr|rr|} 
			\multicolumn{1}{c}{}
			& \multicolumn{2}{c}{\textbf{SCS}}
			& \multicolumn{2}{c}{\textbf{FISTA}}\\
        \hline
			\rowcolor{gray!30}
			%
			\multicolumn{1}{|c|}{$\quad\mathbf{n}$}
			&  \multicolumn{1}{|c|}{\textbf{CPU}}
			& \multicolumn{1}{|c|}{\textbf{It.}}
			&  \multicolumn{1}{|c|}{\textbf{CPU}}
			& \multicolumn{1}{|c|}{\textbf{It.}}\\
    \hline
		 2,000 & 2.67 & 4,400 & \textbf{1.55} & 1,676\\
		 4,000 & 2.69 & 2,175 & \textbf{1.61} & 1,988\\
		 6,000 & 3.48 & 1,750 & \textbf{3.08} & 2,592\\
		 8,000 & 7.01 & 1,875 & \textbf{4.34} & 2,777\\
		 10,000 & 11.41 & 2,075 & \textbf{7.57} & 3,889\\
		 12,000 & 16.66 & 2,350 & \textbf{10.80} & 4,649\\
		 14,000 & 24.19 & 2,875 & \textbf{11.14} & 4,103\\
		 20,000 & 45.30 & 3,750 & \textbf{22.09} & 5,580\\
\hline
		\end{tabular}
        \caption{CPU time in seconds and number of iterations (It.) for SCS \ngg{and for}  
FISTA \ngg{using our closed-form projection.}
\label{table:LogisticRegression1TimeTenthInterest}}
\end{table}

\begin{table}[ht!]
	\centering
		\begin{tabular}{
				|r|rr|rr|} 
			\multicolumn{1}{c}{}
			& \multicolumn{2}{c}{\textbf{SCS}}
			& \multicolumn{2}{c}{\textbf{FISTA}}\\
    \hline
			\rowcolor{gray!30}
			%
			\multicolumn{1}{|c|}{$\quad\mathbf{n}$}
			& \multicolumn{1}{|c|}{\textbf{Obj.}}
			& \multicolumn{1}{|c|}{\textbf{$\|x\|_0$}}
			& \multicolumn{1}{|c|}{\textbf{Obj.}}
			& \multicolumn{1}{|c|}{\textbf{$\|x\|_0$}}\\
            \hline
	 2,000 & 44.830099 & 141 & \textbf{44.800432} & 54\\
	 4,000 & \textbf{43.877616} & 170 & 43.882937 & 61\\
	 6,000 & 43.785346 & 198 & \textbf{43.762135} & 61\\
	 8,000 & 42.420527 & 186 & \textbf{42.394784} & 64\\
\raisebox{0pt}[0pt][0pt]
	 10,000 & \textbf{42.031638} & 207 & 42.036527 & 52\\
	 12,000 & \textbf{41.834190} & 231 & 41.841639 & 62\\
	 14,000 & \textbf{42.085320} & 302 & 42.092437 & 64\\
	 20,000 & 40.119401 & 314 & \textbf{40.111368} & 64\\
\hline
		\end{tabular}
\caption{Objective value (Obj.) and support size ($\|x\|_0$) \ngg{for SCS and for FISTA using our closed-form projection. 
The support size is computed as the number of coordinates of $x$ whose
absolute value exceeds $10^{-4}$.}\label{table:Logistic10PercentObjective}}
\end{table}

The results show a clear performance advantage for FISTA. FISTA consistently achieves lower CPU times than SCS across all tested instances. The number of iterations required by the two methods is of the same order of magnitude. 
The absolute time savings grow with $n$, indicating better scalability for the considered parameter settings. Table~\ref{table:Logistic10PercentObjective} shows the solution quality in terms of objective value and support size $\vectornorm{x}_0$. The two methods attain nearly identical objective
values across all tested instances, indicating solutions of comparable
quality. The support sizes produced by the methods (as reflected by the chosen $10^{-4}$ cutoff) differ, with FISTA 
yielding sparser solutions. 

\section{Conclusions and Future Work}

We derived and proved a closed-form expression for the problem of projection onto a capped rotated second-order cone. The closed-form projection is most useful within first-order methods (projected or proximal gradient, FISTA) whenever the feasible set is a Cartesian product of capped RSOCs, as in perspective relaxations with indicator variables. Each projection requires $O(n)$ operations, so the per-iteration cost grows linearly with the problem dimension.

Utilizing the closed-form projection solution within standard projection-based algorithms will enable the efficient solution of sparse statistical learning problems. Another line of work that we are pursuing is to  
		extend projection-based operator splitting methods (e.g., ADMM), similar to the one proposed in~\cite{ODonoghue_Boyd_2016}, to make use of our rapid projection computations to solve the perspective relaxation of more elaborate linearly constrained formulations (such as cardinality constrained portfolio optimization problems) or as an alternative approach for solving sparse SVM problems.

	\section*{Data availability}
    The datasets were generated synthetically using scikit-learn’s make\_classification as described in Section 3; the code is available from the authors on request.

	\bibliography{cq}
	\bibliographystyle{abbrv}
	
	\appendix
	

    \appendix
\section{Solution of the Depressed Cubic Equation}
\label{sec:depressed_cubic_solution}

Consider the depressed cubic equation
\begin{equation}
\label{eq:depressed_cubic}
x^3+px+q=0,
\end{equation}
where $x$ is the unknown and $p,q\in\mathbb{R}$ are given constants. According to \cite[Section~5.6]{winkler93}, the roots are classified according to the
discriminant, defined by
\begin{equation*}
\label{eq:cubic_discriminant}
\Delta
=
-\left(\frac{q}{2}\right)^2
-
\left(\frac{p}{3}\right)^3 .
\end{equation*}
Note that our sign convention for $\Delta$ is the negative of the
standard discriminant convention for the depressed cubic.
If \underline{$\Delta<0$}, then \eqref{eq:depressed_cubic} has one real root and two conjugate complex
 roots. The real root is
\begin{equation}\label{generalForm_SingleRoot}
x
=
\sqrt[3]{-\frac{q}{2}+\sqrt{-\Delta}}
+
\sqrt[3]{-\frac{q}{2}-\sqrt{-\Delta}} .
\end{equation}

If \underline{$\Delta=0$}, then all roots are real and at least two roots coincide. In this case,
\begin{equation*}
\label{eq:cubic_delta_zero}
x_1 = 2\sqrt[3]{-\frac{q}{2}},
\qquad
x_2=x_3=-\sqrt[3]{-\frac{q}{2}} .
\end{equation*}

If \underline{$\Delta>0$}, then \eqref{eq:depressed_cubic} has three distinct real roots. In this
case, necessarily $p<0$. Define
\begin{equation*}
\label{eq:cubic_theta}
\theta
=
\arccos\left(
\frac{3q}{2p}\sqrt{\frac{-3}{p}}
\right).
\end{equation*}
Then the three real roots are
\begin{equation*}
\label{eq:cubic_three_roots}
x_k
=
2\sqrt{\frac{-p}{3}}
\cos\left(
\frac{\theta+2\pi k}{3}
\right),
\qquad
k=0,1,2 .
\end{equation*}
\paragraph{Remark} In \cite[Formula~5.6.12]{winkler93}, the expression for the case of three-real-roots appears with the opposite overall sign. The formulas presented here are the corrected version, consistent with the depressed cubic \eqref{eq:depressed_cubic}.

\section{Solution of the Cubic Equation}
\label{app:general_cubic}

Consider the cubic equation
\begin{equation}
\label{eq:general_cubic}
\alpha^{3}
+\left(2u^{2}-2u\hat{y}\right)\alpha
-2u^{2}\|\hat{x}\|
=0,
\end{equation}
where $\alpha$ is the unknown and
$u>0$, $\hat{x}\in\mathbb{R}^{n-2}$ and $\hat{y}\in\mathbb{R}$ are given constants.

Equation \eqref{eq:general_cubic} is a depressed cubic of the form
\eqref{eq:depressed_cubic} considered in Appendix~\ref{sec:depressed_cubic_solution}, where
\[
p=2u^{2}-2u\hat{y},
\qquad
q=-2u^{2}\|\hat{x}\|.
\]

The (negative of the) discriminant of
\eqref{eq:general_cubic} is
\[
\Delta
=
-\left(\frac{q}{2}\right)^2
-
\left(\frac{p}{3}\right)^3
=
\frac{-u^3(8u^3-24\hat{y}u^2+(24\hat{y}^2+27\|\hat{x}\|^2)u-8\hat{y}^3)}{27}.
\]
According to Appendix~\ref{sec:depressed_cubic_solution}, the solutions are classified as follows.
If \underline{$\Delta<0$}, then the equation has a unique real root,
\begin{align}
\alpha
&=
\frac{\sqrt{2u}}{6}
\left(
54\sqrt{2u}\|\hat{x}\|
+
6\sqrt{
162u\|\hat{x}\|^{2}
-48\hat{y}^{3}
+144u\hat{y}^{2}
-144u^{2}\hat{y}
+48u^{3}}
\right)^{\frac13}
\nonumber\\
&\quad
+\sqrt{2u}\left(-2u+2\hat{y}\right)
\left(
54\sqrt{2u}\|\hat{x}\|
+
6\sqrt{
162u\|\hat{x}\|^{2}
-48\hat{y}^{3}
+144u\hat{y}^{2}
-144u^{2}\hat{y}
+48u^{3}}
\right)^{-\frac13}.
\label{eq:alpha_closed_form}
\end{align} Otherwise, if \underline{$\Delta=0$}, then the equation has three real roots $\alpha_1,\alpha_2,\alpha_3$,
two of which coincide. The three real roots are
\[
\alpha_{1}=2\sqrt[3]{u^{2}\|\hat{x}\|},
\qquad
\alpha_{2}=\alpha_{3}=-\sqrt[3]{u^{2}\|\hat{x}\|}.
\]
Finally, if \underline{$\Delta>0$}, then \ngg{$\bigl(\tfrac{p}{3}\bigr)^3 < -\bigl(\tfrac{q}{2}\bigr)^2 \le 0$,
so $p=2u(u-\hat y)<0$ and $\hat y>u$. Hence $2u\hat y-2u^2>0$, the
argument of the arccos below lies in $(-1,1)$, and the equation has
three distinct real roots given by the trigonometric expressions} 
\begin{align*}
\theta &=&&\arccos\!\left(u^2\|\hat{x}\|\sqrt{\frac{27}{(2u\hat{y}-2u^2)^3}}\right), && && 
\alpha_1 &=&&2\sqrt{\frac{2u\hat{y}-2u^{2}}{3}}\cos\!\left(\frac{\theta}{3}\right),\\
\alpha_2&=&&2\sqrt {\frac{2u\hat{y}-2u^{2}}{3}}\cos\!\left(\frac{\theta+2\pi}{3}\right), && && 
\alpha_3&=&&2\sqrt {\frac{2u\hat{y}-2u^{2}}{3}}\cos\!\left(\frac{\theta-2\pi}{3}\right).
\end{align*}
Note that Lemma~\ref{lem:conditions-equivalence}-\ref{lem:conditions-equivalence:a} identifies a unique positive root $\alpha$ in this case as well. 

\section{Logistic Regression Formulation for SCS}\label{sec:LogRegForSCS}
{To solve the continuous perspective relaxation~\eqref{form:nonlinlogreg} using SCS,
the problem must be expressed in conic form, with a linear
objective and constraints represented using cones supported by
SCS. Since the logistic loss $L$ in~\eqref{form:logisticUsingOurObject} contains logarithmic and
exponential terms, it cannot be used directly in this form.

The constraints $(x_j,y_j,\bar{z}_j)\in\mathcal{P}$ can be implemented in SCS 
as the 
second-order cone constraint 
\[
\|(\sqrt{2}x_j,\; y_j-\bar{z}_j)\|_2 \le y_j+\bar{z}_j,
\]
together with a bound constraint $0\le\bar{z}_j\le 0.5$ for each $j\in [n]$. 
The epigraph of the soft-plus function $f(s)=\log\left(1+\exp(s)\right)$ (appearing in the loss function $L$) admits an exponential cone representation \ishyComments{\cite[Eq.~33]{ODonoghue_Boyd_2016}}. So, to obtain a formulation expressed entirely in terms of standard conic constraints, we represent the soft-plus epigraph using exponential-cone constraints. The exponential cone is defined by
\begin{equation*}
\begin{aligned}
K_{\exp}
={}&\{(r,s,\tau)\in\mathbb{R}^3 \mid
      s>0,\; s\exp(r/s)\le \tau\} 
\cup
\{(r,0,\tau)\in\mathbb{R}^3 \mid
      r\le0,\; \tau\ge0\}.
\end{aligned}
\end{equation*} Using the exponential cone representation of the epigraph of the soft-plus function, introducing the auxiliary variables $d_1,d_2,d_3,d_4$, 
formulation \eqref{form:nonlinlogreg} can be written as
	\begin{align*}
		&\min_{\substack{x, y, \bar{z} \in \mathbb{R}^n,\\ d_1, d_2, d_3, d_4 \in \mathbb{R}^t}} 
		&& d_4^\top \mathbf{1} + \bar{\gamma}_0 \sum_{j=1}^n \bar{z}_j + \gamma_2 \sum_{j=1}^n y_j \\
		&\text{s.t.} 
		&& (1-2b)\circ A x = d_1, \\
		& && d_2 + d_3 \leq \mathbf{1} \\
		& && (x_j,y_j,\bar{z}_j)\in \mathcal{P}, && j \in [n]\\
		& && \left(d_{1i}-d_{4i},\,1,\,d_{2i}\right)\in K_{\exp}
  && i\in [t]\\
		& && \left(-d_{4i},\,1,\,d_{3i}\right)\in K_{\exp}
&& i\in [t],
    \end{align*}
\ngg{where $\circ$ denotes an elementwise product and $\mathbf{1}$ is the all-ones vector of corresponding dimension. Since $L(x) = \sum_{i=1}^t f((1-2b_i)A_ix)$, the variable $d_{4i}$ bounds the $i$-th loss term.}  
Consequently, the epigraph of the soft-plus function is represented entirely through affine, second-order cone, and exponential-cone constraints, making the formulation 
compatible with 
conic optimization solvers such as SCS. 

\section{FISTA for the Perspective Relaxation of Sparse Logistic Regression}
\label{app:fista}

Accelerated first-order methods improve the convergence rate of
standard gradient and proximal-gradient methods by incorporating a
momentum term; see Nesterov~\cite{Nesterov1983} and the
surveys in~\cite{Dvurechensky2021,dAspremont2021}.
Beck and Teboulle~\cite{Beck09} proposed the fast iterative shrinkage-thresholding algorithm (FISTA), an accelerated
proximal-gradient method for composite convex optimization problems
of the form
\begin{align}\label{FISTA_form}
\min_w \; f(w)+g(w),    
\end{align}
where $f$ is convex and differentiable with an $L$-Lipschitz
continuous gradient, and $g$ is proper, closed, and convex with an
efficiently computable proximal operator. To apply FISTA to formulation~\eqref{form:nonlinlogreg}, let
$w=(x,y,\bar z)$. By incorporating the constraints into the
objective through indicator functions, formulation~\eqref{form:nonlinlogreg} can be
equivalently written in the form of \eqref{FISTA_form} letting
\[
f(x,y,\bar{z}) \;=\; L(x)
+\bar{\gamma}_0\sum_{j=1}^{n}\bar z_j
+\gamma_2\sum_{j=1}^{n}y_j,
\] 
and \ngg{the nonsmooth term is a sum of indicator functions,}
\[
g(x,y,\bar{z})=\sum_{j=1}^{n}\mathbb{I}_{\mathcal{P}}(x_j,y_j,\bar{z}_j).
\]
Hence, the proximal step in each FISTA iteration decomposes into
$n$ independent projections onto $\mathcal{P}$, which are computed
using Theorem~\ref{opt979}. Algorithm~\ref{alg:fista} summarizes the FISTA iteration used in our
implementation, including the stopping criterion based on the
difference between consecutive iterates. \ishyComments{Although Algorithm~\ref{alg:fista} is presented using a fixed Lipschitz constant \(L\), in our implementation \(L\) is estimated using the toolbox's default backtracking line search.}
\begin{algorithm}[H]
	\caption{Fast Iterative Shrinkage-Thresholding Algorithm (FISTA)~\cite{Beck09}}
	\label{alg:fista}
	\begin{algorithmic}[1]
		\Require initial point $w^{0}$, Lipschitz constant $L>0$, tolerance $\epsilon$.
		\State $\bar{w}^{1} \gets w^{0}$
		\State $\theta_{1} \gets 1$
		\For{$k = 1,2,\ldots$}
		\State $w^{k} \gets \operatorname{prox}_{\tfrac{1}{L}g}\!\left(\bar{w}^{k} - \tfrac{1}{L}\nabla f(\bar{w}^{k})\right)$
		\State $\theta_{k+1} \gets \tfrac{1+\sqrt{1+4\theta_{k}^2}}{2}$
		\State $\bar{w}^{k+1} \gets w^{k} + \frac{\theta_{k}-1}{\theta_{k+1}}\,(w^{k}-w^{k-1})$
		\If{$\vectornorm{w^{k}-w^{k-1}}<\epsilon$}
		\State \textbf{break}
		\EndIf
		\EndFor
        \State \Return $w^k$
	\end{algorithmic}
\end{algorithm}
\end{document}